\documentclass[12pt, reqno]{amsart}
\usepackage{amsmath, amsthm, amscd, amsfonts, amssymb, graphicx, xcolor}
\usepackage{todonotes}
\usepackage[bookmarksnumbered, colorlinks, plainpages]{hyperref}
\usepackage{tikz}
\usepackage{tikz-3dplot} 
\usetikzlibrary{calc} 
\usepackage{tikz-3dplot}
\usetikzlibrary{patterns}

\newtheorem{theorem}{Theorem}[section]
\newtheorem{lemma}[theorem]{Lemma}
\newtheorem{proposition}[theorem]{Proposition}

\theoremstyle{definition}

\newtheorem{example}[theorem]{Example}

\theoremstyle{remark}
\newtheorem{remark}[theorem]{Remark}
\numberwithin{equation}{section}

\newcommand{\interior}[1]{%
  {\kern0pt#1}^{\mathrm{o}}%
}

\begin{document}
\setcounter{page}{1}

\color{darkgray}{
\noindent 

\centerline{}

\centerline{}

\newcommand{\pknote}[1]{\todo[inline, color=cyan]{PK:\ #1}}
\newcommand{\term}[1]{\emph{#1}}

\newcommand{\wwnote}[1]{\todo[inline, color=yellow]{WW:\ #1}}


\title[Generalization of Ceva theorem]{Generalization of Ceva theorem}

\author[Wojciech Wdowski]{Wojciech Wdowski}




\begin{abstract}
In this paper, we present a novel generalization of the classical Ceva theorem to arbitrarily dimensional simplexes. Our approach allows cevians to have any dimension (smaller than the dimension of the base simplex). Consequently, our result unifies other generalizations of the Ceva theorem obtained in recent years.
\end{abstract} \maketitle

\section{Introduction}

\subsection{Ceva's Theorem}\label{Ceva's theorem}
Ceva's Theorem---first proved by Yusuf al-Mu'taman ibn Hud, a king of~Zaragoza---is a fundamental result in the classical $2$-dimensional geometry. Although we firmly believe that it is well known to the readers, we nonetheless recall it here to make the whole exposition cleaner.

\begin{theorem}[Ceva's Theorem]
  Let $\Delta P_0P_1P_2$ be a triangle, and let points $Q_0,Q_1,Q_2$ lie on the edges opposite to vertices $P_0,P_1,P_2$, respectively \textup(see Figure~\ref{fig:trojkat_punkty_Q}\textup). Then, simplices $P_iQ_{i}$ \textup(for $i\in\{0,1,2\}$\textup) intersect nonempty if and only if the following formula holds
  \[
  \frac{P_0Q_1}{Q_1P_2}\frac{P_2Q_0}{Q_0P_1}\frac{P_1Q_2}{Q_2P_0}=1.
  \]
\end{theorem}

Recall that the segment $P_iQ_{i}$ is called a \term{cevian}, and point $Q_{i}$ is called \term{cevian's foot} (for $i\in\{0,1,2\}$).
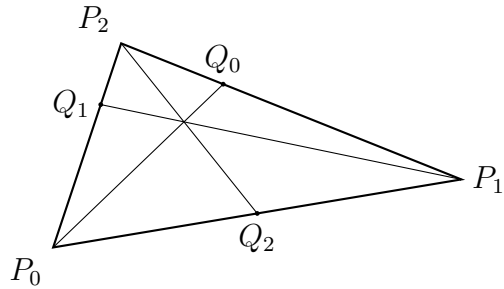
\begin{figure}\label{fig:trojkat_punkty_Q}
  \centering
  \begin{tikzpicture}[scale=0.9]
  \coordinate (P0) at (0,0);
  \coordinate (P1) at (6,1);
  \coordinate (P2) at (1,3);

  \node[below left] at (P0) {$P_0$};
  \node[right] at (P1) {$P_1$};
  \node[above left] at (P2) {$P_2$};
  
  \coordinate (Q0) at ($ (P1)!0.7!(P2) $);
  \coordinate (Q1) at ($ (P2)!0.3!(P0) $);
  \coordinate (Q2) at ($ (P0)!0.5!(P1) $);

  \fill (Q0) circle (1pt) node[above] {$Q_0$};
  \fill (Q1) circle (1pt) node[left] {$Q_1$};
  \fill (Q2) circle (1pt) node[below] {$Q_2$};

  \draw (P0) -- (Q0);
  \draw (P1) -- (Q1);
  \draw (P2) -- (Q2);

  \draw[thick] (P0) -- (P1) -- (P2) -- cycle;

  \end{tikzpicture}
  \caption{Demonstration of Ceva's Theorem.}
\end{figure}

This paper aims to extend the above result to simplices of arbitrary dimensions. Prior work on this topic has been conducted in \cite{Buba-Brzozowa, Prince, Samet, Witczynski}. In all four papers, the authors used an $n$-dimensional simplex as a (natural) generalization of a triangle. However, the notion of a cevian differs between the papers. Witczyński and Buba-Brzozowa (see \cite{Witczynski,Buba-Brzozowa}) use $(n-1)$-dimensional simplices to generalize the concept of a cevian. On the other hand, in Samet's paper \cite{Samet}, cevians are simply $1$-dimensional. These two interpretations exhaust the possibilities for defining cevians in three-dimensional space. However, in higher dimensions, there are other possibilities. For example, when the dimension equals~$4$, one may think of $2$-dimensional cevians and ask whether an analogous result holds.

Below, we present a result (see Theorem~\ref{Main!}) that generalizes the Ceva theorem to simplices of an arbitrary dimension $n\geq 2$ and cevians of dimension~$k$ for any $1 \leq k < n$. Our result encompasses and extends those of Witczyński, Buba-Brzozowa, and Samet. In fact, the above-mentioned theorems appear as border cases of our main result.

Similar results have already been obtained in \cite{Prince}, by means of algebraic geometry techniques. In contrast, the present work uses a much simpler method to obtain the same conclusions, while also provide a straightforward condition for every $k\in\{1,...,n\}$. 

For a complete understanding of this paper (besides a working knowledge of $n$-dimensional real space), we assume familiarity with the framework developed in \cite{Samet}.

\subsection{Simplices}\label{symplex}
Let $P_0P_1\dots P_n$ be points in a general position of an $n$-dimensional Euclidean space. We denote by $\mathbf{S} = \Delta P_0P_1\dotsc P_n$, the $n$-dimension simplex spanned by these points. The \term{interior} of the simplex $\mathbf{S}$, denoted~$\interior{\mathbf{S}}$, is the set of points of the form 
\[
\alpha_0P_0+\dotsb +\alpha_nP_n, 
\]
where $\sum^n_{k=0}\alpha_k=1$ and $\alpha_k>0$ for every~$k$.

Any subset of $k+1$ points from $\{P_0,P_1,\dots,P_n\}$ spans a $k$-dimensional \term{subsimplicial face} of~$\mathbf{S}$, which we refer to as a $k$-\term{face} (or simply \term{face}). In particular, when $k = n - 1$, the resulting $(n-1)$-face is called a \term{facet}. We denote by $\mathbf{S}_i$ the facet of~$\mathbf{S}$ opposite to the vertex~$P_i$, i.e., the facet that does not contain~$P_i$. 

\subsection{Construction of $k$-dimension cevians}\label{construction}
We begin by defining the notion of $k\mbox{-cevians}$. Assume that the simplex~$\mathbf{S}$ is fixed once and for all. Take a subset $U = \{u_1,\dotsc ,u_k\}$ of the set of indices $\{0,\dotsc ,n\}$, where $0\leq u_1<u_2<\dotsb <u_k\leq n$. Denote its completion by $U' = \{0,\dotsc ,n\}\setminus U$, and assume that $U'=\{v_1,\dotsc ,v_l\}$, where $0\leq v_1<v_2<\dotsb <v_{l}\leq n$ and $k+l=n+1$. Next, select a point $Q_{U'}$ in the interior of $\Delta P_{v_1}\dotsc P_{v_l}$. The simplex $\Delta Q_{U'}P_{u_1}\dotsc P_{u_k}$ will be called a \term{$k\mbox{-cevian}$} of $\mathbf{S}$. 

The point $Q_{U'}$ will be referred to as the \term{foot} of the $k$-cevian. In an $n\mbox{-dimensional}$ simplex, there are ${n+1 \choose k}$ possible choices for the set~$U$. Assume that for each such choice, we have selected a corresponding foot $Q_{U'}$. This way, we obtain a family of ${n+1 \choose k}$ $k\mbox{-cevians}$, which we denote by~$\mathfrak{R}$. 

Theremainder of the paper will be devoted to finding conditions on simplices that ensure that the family~$\mathfrak{R}$ has a nonempty intersection In what follows, the family~$\mathfrak{R}$ of $k\mbox{-cevians}$, like the simplex~$\mathbf{S}$, will be assumed to be fixed.


\subsection{\texorpdfstring{$l$-Faces}{l-Faces}}
An essential role in the subsequent discussion will be played by the $l$-dimensional faces of~$\mathbf{S}$, where $k + l = n + 1$. Accordingly, we will study the relationship between the $k$-cevians and the $l$-faces.

Let $\mathbf{L}$ be an $l$-face of the form  
\[
\mathbf{L}=\Delta P_{j_0}\dotsc P_{j_l},
\]
where $J=\{j_0,\dotsc ,j_l\}\subset\{0,1,...,n\}$. Let $J' = \{ i_1, \dotsc, i_{k-1} \} = \{ 0, \dotsc, n \} \setminus J$ denote the complement of $J$.

Observe that each \emph{facet} of~$\mathbf{L}$---that is, each $(l-1)$-dimensional face of $\mathbf{L}$---contains exactly one foot of a $k$-cevian from the family~$\mathfrak{R}$. Specifically, for each $t\in J$, the interior of the facet ${\mathbf{L}_t}$, spanned by $\{ P_j: j\in J, j\neq t\}$, contains the foot~$Q_{J\backslash\{t\}}$ of a $k$-cevian
\begin{equation} \label{form}
C_t = \Delta Q_{J\backslash\{t\}} P_t P_{i_1} \dotsc P_{i_{k-1}}.
\end{equation}
Moreover, every $k$-cevian whose foot lies in a facet of~$\mathbf{L}$ is of the form \eqref{form}. Thus, the subfamily of $\mathfrak{R}$ consisting of $k$-cevians whose feet lie in~$\mathbf{L}$ has precisely $l+1$ elements. We denote this subfamily by~$\mathfrak{R}_\mathbf{L}$.


Throughout the rest of the paper, while analyzing the cevian feet within a given $l$-face, we shall adopt the notation $Q_{[t]}$ in place of $Q_{J\backslash\{t\}}$.


\subsection{Induced cevians}
Given an $l$-face~$\mathbf{L}$ of~$\mathbf{S}$ and a $k$-cevian $C_t\in\mathfrak{R}_\mathbf{L}$ with foot~$Q_{[t]}$ in the interior of~$\mathbf{L}$ we define the \term{induced cevian}~$c_t$ in~$\mathbf{L}$ as the intersection 
\[
c_t = C_t\cap\mathbf{L} = \Delta Q_{[t]}P_t.
\]
In this way, we obtain a family of $1$-cevians in~$\mathbf{L}$, defined as in~\cite{Samet}, which we denote by $\mathfrak{r}_\mathbf{L}$. 

Before proceeding further, we present two examples that illustrate the notions introduced above and place them in the proper context with respect to prior work.
  
\begin{example}[Case $n=3\text{ and }k=2$]
Let $\mathbf{S} = \Delta P_0P_1P_2P_3$ be a simplex. Construct a $2$-cevian in~$\mathbf{S}$. To this end take a set $U=\{2,3\}$ (so $U'=\{0,1\}$), and select a point $Q_{U'}\in\interior{\Delta P_0P_1}$. Then the simplex $C = \Delta Q_{U'}P_2P_3$ is a $2\mbox{-cevian}$ (see Figure~\ref{n=3,k=2}).

Since $k=2$, we have that $l=2$. Therefore, $\mathbf{L}=\Delta P_0P_1P_3$ is a $2\mbox{-face}$ of $\mathbf{S}$. Then $Q_{U'}=Q_{[3]}$,
\[
C=C_3=\Delta Q_{[3]}P_3P_2
\]
and the induced cevian~$c_3$ in~$\mathbf{L}$ has the form $c_3=\Delta Q_{[3]}P_3$. It is a line segment, and so a cevian in the most classical sense.

Moreover, as $k=n-1$, the $2$-cevian~$C$ is a cevian in the sense of \cite{Witczynski} and \cite{Buba-Brzozowa}.
\end{example}

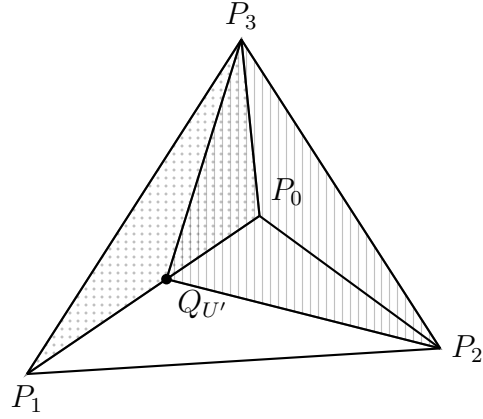
\begin{figure}  
\centering  
\tdplotsetmaincoords{70}{110} 
\begin{tikzpicture}[tdplot_main_coords, scale=4]  

\coordinate (P0) at (-0.5,0.2,0);  
\coordinate (P1) at (1.2,0,0);  
\coordinate (P2) at (0.5,1.2,0);  
\coordinate (P3) at (0.5,0.5,1);  

\draw[thick] (P0) -- (P1);  
\draw[thick] (P0) -- (P2);  
\draw[thick] (P1) -- (P2);  
\draw[thick] (P0) -- (P3);  
\draw[thick] (P1) -- (P3);  
\draw[thick] (P2) -- (P3);  

\filldraw[opacity=0.5, thick, pattern=dots] (P0) -- (P1) -- (P3) -- cycle;  

\node at (P0) [above right] {$P_0$};  
\node at (P1) [below] {$P_1$};  
\node at (P2) [right] {$P_2$};  
\node at (P3) [above] {$P_3$};  

\path (P0) -- (P1) coordinate[pos=0.4] (QUprime);  
\node[below right] at (QUprime) {$Q_{U'}$}; 

\draw[thick] (QUprime) -- (P2);  
\draw[thick] (QUprime) -- (P3);  

\filldraw[opacity=0.5, thick, pattern=vertical lines] (P2) -- (P3) -- (QUprime) -- cycle;  
\fill (QUprime) circle (0.5pt); 
\end{tikzpicture}  
\caption{\label{n=3,k=2}Case $n=3, k=2$ with the $2\mbox{-cevian } C$ is dashed and the $2\mbox{-face } \mathbf{L}$ is dotted.}  
\end{figure}


\begin{example}[Case $n=4$, $k=2$]
Now consider a $4$-dimensional simplex $\mathbf{S} = \Delta P_0P_1P_2P_3P_4$. We again construct a $2$-cevian in $\mathbf{S}$. Let $U=\{2,4\}$, so its complement is $U' = \{0,1,3\}$. Choose a point $Q_{U'}$ in the interior of the subsimplex $\Delta P_0P_1P_3$. Then the simplex $C=\Delta Q_{U'}P_2P_4$ (see Figure~\ref{n=4,k=2}) is a $2$-cevian in~$\mathbf{S}$.

In this case, $k=2$ and $l=3$, so $\mathbf{L}=\Delta P_0P_1P_2P_3$ is a $3$-dimensional face of~$\mathbf{S}$. Then $Q_{U'}=Q_{[2]}$,
\[
    C=C_2=\Delta Q_{[2]}P_2P_4
\]
and the induced cevian~$c_2$ is of the form  
\[
c_2 = \Delta Q_{[2]} P_2. 
\]

It should be emphasized that the cevian $C$ is neither a line segment nor a simplex of codimension one. To the best of the author's knowledge, such cases have not been treated in previous literature.
\end{example}

\begin{figure}
\centering
\tdplotsetmaincoords{70}{110} 
  \begin{tikzpicture}[tdplot_main_coords, scale=4]

      \coordinate (P0) at (-0.5,0.2,0);
      \coordinate (P1) at (1.2,0,0);
      \coordinate (P2) at (0.5,1.2,0);
      \coordinate (P3) at (0.5,0.5,1);
    
      \coordinate (P4) at (-1, 1, 0.5);
      \node at (P4) [above left] {$P_4$};

        \fill[black] (P4) circle (0.4pt); 
        
      \draw[thick] (P0) -- (P1);
      \draw[thick] (P0) -- (P2);
      \draw[thick] (P1) -- (P2);
      \draw[thick] (P0) -- (P3);
      \draw[thick] (P1) -- (P3);
      \draw[thick] (P2) -- (P3);
    
      \node at (P0) [above right] {$P_0$};
      \node at (P1) [below] {$P_1$};
      \node at (P2) [right] {$P_2$};
      \node at (P3) [above] {$P_3$};

    \draw[dashed, thick] (P4) -- (P0);
    \draw[dashed, thick] (P4) -- (P2);
    \draw[dashed, thick] (P4) -- (P3);

    \filldraw[opacity=0.2, thick, pattern=dots] (P0) -- (P1) -- (P2) -- cycle; 
    \filldraw[opacity=0.2, thick, pattern=dots] (P0) -- (P1) -- (P3) -- cycle; 
    \filldraw[opacity=0.2, thick, pattern=dots] (P0) -- (P2) -- (P3) -- cycle; 
    \filldraw[opacity=0.2, thick, pattern=dots] (P1) -- (P2) -- (P3) -- cycle; 
    
    \draw[dashed, thick] (P4) .. controls (-0.5, 2, 0.4)  and (0.5, 1.5, -0.6) .. (P1);

    \coordinate (QUprime) at ($(P0)!0.33!(P1)!0.3!(P3)$); 
    \fill (QUprime) circle (0.5pt); 
    \node[below] at (QUprime) {$Q_{U'}$}; 

    \draw[black, thick] (P2) -- (P4) -- (QUprime) -- (P2);
    \filldraw[black, opacity=0.5, thick, pattern=vertical lines] (P2) -- (P4) -- (QUprime) -- cycle; 

  \end{tikzpicture}
  \caption{\label{n=4,k=2} Visualization for Case $n=4, k=2$. The $2\mbox{-cevian } C$ is dashed, and the $3\mbox{-face } \mathbf{L}$ is dotted.} 
\end{figure}
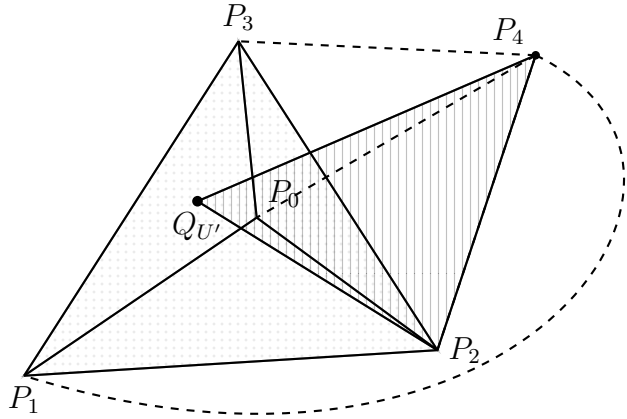

\subsection{Multipedes} 
Finally, we recall the notion of a \term{multipede} and an \term{induced multipede}, introduced originally by Samet in~\cite{Samet}. Let~$\mathbf{T}$ be any $m~\mbox{- dimesional}$ simplex spanned by points from $\{P_0,P_1,...,P_m\}$ (in applications, it will be just a face of~$\mathbf{S}$). Denote by~$\mathcal{F}$ the collection of all faces of~$\mathbf{T}$. A set of points $\{A_{\mathbf{F}}\}_{\mathbf{F}\in \mathcal{F}}$ is called a \term{multipede} if every point $A_\mathbf{F}$ lies in the interior of the corresponding face~$\mathbf{F}$.

We say that a point $A_\mathbf{T}$ \term{induces} a multipede in a simplex~$\mathbf{T}$ if points $A_\mathbf{T}, A_{\mathbf{T}_i}, P_i$ are collinear. Moreover, for every $1\mbox{-face}$ $\mathbf{F}\in\mathcal{F}$ (where $t> 1$) points $A_\mathbf{F}, A_{\mathbf{F}_i}, P_i$ are collinear. Samet proved that each point in the interior of a given simplex induces exactly one multipede. 
Let $\mathcal{M} = \{A_\mathbf{F}\}_{\mathbf{F}\in \mathcal{F}}$ be a multipede of~$\mathbf{T}$ induced by some point~$Q$. Fix a face~$\mathbf{F}\in \mathcal{F}$ of~$\mathbf{T}$ and let $P := A_\mathbf{F}$ be a point of the multipede~$\mathcal{M}$ in the interior of~$\mathbf{F}$. Then the multipede~$\mathcal{N}$ of~$\mathbf{F}$ induced by~$P$ is a subset of~$\mathcal{M}$. This inclusion let us define a relation $\preceq$ on the set of points of~$\mathbf{T}$. We say that $P\preceq Q$ if $P$ is one of the points of the multipede induced by~$Q$ (i.e., if the multipede induced by~$P$ is a subset of the multipede induced by~$Q$). This relation is not well-defined on the vertices of~$\mathbf{T}$. To remedy this, for any vertex~$P_i$ and a point~$Q$ in the interior of any face containing~$P_i$, we set $P_i\preceq Q$.
  
\section{Main results}

\begin{theorem}\label{main}
  Let $\mathfrak{R}$ be family of $k\mbox{-cevians}$ in $\mathbf{S}$.
  The following conditions are equivalent:
  \begin{enumerate}
  \item \label{1} $\underset{X\in\interior{S}}{\exists}X\in\bigcap\mathfrak{R}$ 
  \item \label{2} For every $l\mbox{-face}$ $\mathbf{L}$ we have $\bigcap\mathfrak{r}_\mathbf{L}\neq\emptyset$.
  \end{enumerate}
\end{theorem}

The proof of this theorem must be preceded by a lemma.

\begin{lemma}
Fix an $l$-face $\mathbf{L}$ and let $Q_{[i]}$ and $Q_{[j]}$ be feet of $k$-cevians lying in the interiors of the facets $\mathbf{L}_i$ and $\mathbf{L}_j$, respectively. Further, let $A$ and $B$ be points such that $A\preceq Q_{[i]}$, $B\preceq Q_{[j]}$ and $A,B\in\interior{\mathbf{F}}\subset\mathbf{L}$. Then condition~\eqref{2} of Theorem~\ref{main} implies that $A=B$.
\end{lemma}

\begin{proof}
 Let $X\in\bigcap\mathfrak{r}_\mathbf{L}$ then $Q_{[i]}\preceq X$ and $Q_{[j]}\preceq X$ so $A\preceq X$ and $B\preceq X$ since $A$ and $B$ are both in $\interior{\mathbf{F}}$, by \cite[Proposition~1]{Samet} it follows that $A=B$.
\end{proof}


\begin{remark}\label{important}
 Let $W$ be the set of all points generated by all $k\mbox{-cevians}$ feet. If condition~\eqref{2} holds, then for each $t\mbox{-face}$ (where $1\leq t\leq l-1$) $\mathbf{T}$ we get
 \[
    |W\cup \interior{\mathbf{T}}|=1
 \]
\end{remark}

\begin{remark}\label{important2}
  Condition~\eqref{2} is equivalent to the following one
\begin{itemize}
\item For every $l$-dimensional face $L$, there exists a point~$X$ such that $$Q_{[t]}\preceq X$$ for all $t$, where $Q_{[t]}$ is foot of a $k\mbox{-cevian}$ $C_t\in \mathfrak{R}_\mathbf{L}$ (i.e., $Q_{[t]}$ is a foot of induced cevian $c_t\in \mathfrak{r}_\mathbf{L}$).
\end{itemize}
\end{remark}

We are now in a position to prove the main theorem.

\begin{proof}[Proof of Theorem~\ref{main}]
Suppose that condition~\eqref{1} holds. Without loss of generality, we may assume that~$\mathbf{L}$ has the form $\mathbf{L}=\Delta P_0...P_l$. Then, any cevian~$C_t$ in $\mathfrak{R}_\mathbf{L}$ can be expressed as 
\[
C_t=\Delta Q_{[t]}P_t P_{l+1}...P_{n},
\]
where $t\in \{0,...,l\}$. Take a point $X\in \bigcap\mathfrak{R}$, then $X\in \interior{\mathbf{S}}$ and so $X$ is convex combination of the points $P_0P_1...P_n$ with weights $\alpha_0,...,\alpha_n$, where $\alpha_m>0$ and $\sum_{m=1}^n a_m = 1$. Consequently, 
\[
X=\alpha_0 P_0+\alpha_1P_1+ \dotsb +\alpha_{n}P_n.
\]
On the other hand, $X\in C_t$ and $X\in C_s$ for $s,t\in\{0,...,l\}$. It follows that $X$ has the form
\begin{equation}\label{!1}
\begin{cases}
  X &= \beta_1Q_{[t]}+\beta_2P_t+\alpha_{l+1}P_{l+1}+ \dotsb +\alpha_{n}P_n\\
  X &= \gamma_1Q_{[s]}+\gamma_2P_s+\alpha_{l+1}P_{l+1}+ \dotsb +\alpha_{n}P_n.
\end{cases}
\end{equation}
Notice that the uniqueness of barycentric coordinates implies that all the coefficients that correspond to the points $P_{l+1}, \dotsc, P_{n}$ are the same in all three representations. 

Set $\hat X:=\frac{\beta_1}{\beta_1+\beta_2}Q_{[t]}+\frac{\beta_2}{\beta_1+\beta_2}P_t$. It is clear that $\hat X\in\Delta Q_{[t]}P_t$. We claim that $\hat X\in \Delta Q_{[s]}P_s$. 
Denote $K := \sum_{m=l+1}^n\alpha_mP_m$. Then \eqref{!1} yields
\[
  (\beta_1+\beta_2)\cdot\Bigl(\frac{\beta_1}{\beta_1+\beta_2}Q_{[t]}+\frac{\beta_2}{\beta_1+\beta_2}P_t\Bigr)+K
  =
  (\gamma_1+\gamma_2)\cdot \Bigl(\frac{\gamma_1}{\gamma_1+\gamma_2} Q_{[s]}+\frac{\gamma_2}{\gamma_1+\gamma_2}P_s\Bigr)+K.
\]
Now, $\beta_1+\beta_2=1-\alpha_{l+1}+ \dotsb +\alpha_{n}=\gamma_1+\gamma_2$ and so we have
\[
  (\beta_1+\beta_2)(\frac{\beta_1}{\beta_1+\beta_2}Q_{[t]}+\frac{\beta_2}{\beta_1+\beta_2}P_t)+K
  =({\beta_1+\beta_2})(\frac{\gamma_1}{\gamma_1+\gamma_2} Q_{[s]}+\frac{\gamma_2}{\gamma_1+\gamma_2}P_s)+K.
\]
This implies that
\[
  \hat X=\frac{\gamma_1}{\gamma_1+\gamma_2} Q_{[s]}+\frac{\gamma_2}{\gamma_1+\gamma_2}P_s.
\]
Consequently, $\hat{X}\in\Delta Q_{[s]}P_s$ as claimed. This proves that $\bigcap\mathfrak{r}_\mathbf{L}$ is not empty.

Conversely, assume that condition~\eqref{2} holds. If the dimension~$n$ of the simplex~$\mathbf{S}$ is either $2$ or $3$, then the implication $\eqref{2} \Rightarrow \eqref{1}$ is already known. Specifically, it follows from the classical Ceva’s theorem when $n = 2$, and from its generalizations in the case $n = 3$: either from \cite{Buba-Brzozowa, Witczynski} when $k = 2$, or from \cite{Samet} when $k = 1$. Therefore ,we may assume that $n \geq 4$. The proof proceeds by induction on the dimension $k$ of the cevians involved. The base case $k = n-1$ is established in \cite{Buba-Brzozowa}. Assuming the result holds for some $1 < k < n$, we aim to show that it also holds for $k - 1$.
  
Let $\mathfrak{R}^{k-1}$ be a family of $(k-1)$-cevians in $\mathbf{S}$ satisfying~\eqref{2}. We shall construct a family~$\mathfrak{R}$ consisting of $k\mbox{-cevians}$. By Remark~\ref{important} there are precisely ${n+1 \choose l}={n+1 \choose n+1-k}={n+1 \choose k}$ points generated by feet of $(k-1)\mbox{-cevians}$ in the interiors of each $(l-1)\mbox{-face}$.  Let $Q_{U'}$ lie in the interior of an $(l-1)\mbox{-face}$ $\mathbf{F}$ of the form 
\[
\mathbf{F}=\Delta P_{v_1}...P_{v_l},
\]
where $0\leq v_1<v_2< \dotsb <v_{l}\leq n$. Take $U= \{0,\dotsc ,n\}\setminus\{v_1,v_2,\dotsc, v_{l}\}$, say $U=\{u_{1},...,u_k\}$ and $0\leq u_1<u_2< \dotsb <u_{k}\leq n$. Then the simplex 
\[
C=\Delta Q_{U'}P_{u_1}...P_{u_k}
\]
is a $k\mbox{-cevian}$. Repeating this procedure for different points $Q_{U'}$ we obtain a family~$\mathfrak{R}$ consisting of  $n+1 \choose k$ $k\mbox{-cevians}$. It follows from Remark~\ref{important2}, that $\mathfrak{R}$, satisfies~\eqref{2}. Hence, by the inductive hypothesis, there exists a point $X\in \mathfrak{R}$. 
  
To conclude the proof, we need to show $X$ sits in~$\mathfrak{R}^{k-1}$. To this end, take a  $(k-1)$-cevian $\hat{C}=\Delta\hat{Q}P_{i_1}\dotsc P_{i_{k-1}}$, where $\hat{Q}\in\interior{\mathbf{L}}$ is a foot of $\hat{C}$ and $\mathbf{L}={\Delta P_{j_0}\dotsc P_{j_l}}$. Every point in $\hat{C}$ can be expressed as a convex combination 
\[
\hat \omega \hat Q+\omega_1P_{j_1}+ \dotsb +\omega_{k-1}P_{k-1}.
\]
We claim that the point~$X$ also has this form. Since $X\in\interior{\mathbf{S}}$, it~follows that it can be expressed as
\[
X=\alpha_0P_0+\alpha_1P_1+ \dotsb +\alpha_nP_n,
\]
where $\alpha_m>0$ and $\sum^n_{m=0}a_m=1$. Moreover, the fact that $X$ sits in $\mathfrak{R}$ implies that $X\in C_t$, where 
\[
C_t=Q_{[t]}P_tP_{i_1}\dotsc P_{i_{k-1}},
\]
for some $Q_{[t]}\in\mathbf{L}_t$ and $t\in\{j_0,\dotsc ,j_l\}$. Consequently, we have 
\[
X=\beta_1Q_{[t]}+\beta_2P_t+\alpha_{i_1}P_{i_{1}}+\dotsb +\alpha_{i_{k-1}}P_{i_{k-1}}.
\]
Thus, $\hat{Q}=\frac{\beta_1}{\beta_1+\beta_2} Q_{[t]}+\frac{\beta_2}{\beta_1+\beta_2}P_t$ and so putting $\hat \omega=\beta_1+\beta_2$ and $\omega_m=\alpha_{i_{m}}$, we obtain $X\in \hat{C}$. This concludes the proof.
\end{proof}

The preceding result can now be reformulated in a manner analogous to the classical Ceva’s theorem.

\begin{theorem} \label{Main!}
Let $\mathfrak{R}$ be a family of $k\mbox{-cevians}$ in $\mathbf{S}$. The following conditions are equivalent:
  \begin{enumerate}
  \item \label{1!} The family $\mathfrak{R}$ intersects nonempty.
  \item \label{2!} For every $l\mbox{-face}$ $\mathbf{L}$ and every uniform or mixed fan $\{(F_z,L_z,M_z)| z\in\mathbb{Z}/m\mathbb{Z}\}$ one has
  $$\prod_{z\in\mathbb{Z}/m\mathbb{Z}} \frac{\mbox{Vol}(L_z)} {\mbox{Vol}(M_z)}=1$$
  \end{enumerate}
\end{theorem}

  
\begin{proof}
The second conditions of Theorems~\ref{main} and~\ref{Main!} are equivalent, as shown in \cite{Samet}.
\end{proof}

Finally, let us point out that the results in \cite{Buba-Brzozowa, Witczynski, Samet} as well as the classical Ceva's theorem are special cases of Theorem~\ref{Main!}. More specifically:
\begin{enumerate}
  \item if $n=2$ and $k=1$, then Theorem~\ref{Main!} is just the Ceva's theorem;
  \item if $n=3$ and $k=2$, then Theorem~\ref{Main!} is the Witczyński's theorem (see \cite[Proposition~2]{Witczynski});  
  \item if $n>2$ and $k=n-1$, then Theorem~\ref{Main!} boils down to Buba-Brzozowa's theorem (see \cite[Theorem~1]{Buba-Brzozowa}); 
  \item if $n>2$ and $k=1$, then Theorem~\ref{Main!} reduces to Samet's theorem (see \cite{Samet}).
\end{enumerate}


\section{Future work}
Our prior analysis has focused on families of $k\mbox{-cevians}$ of uniform dimension $k$. We can, however, consider families of cevians with different dimensions and prove a Ceva-type theorem in such cases.

\begin{proposition}\label{prop}
    Let $\Delta P_0P_1P_2P_3$ be a tetrahedron. Let us consider the family 
    \[
    \mathfrak{R}=\{\Delta Q_{\{0,1\}}P_2P_3, \Delta Q_{\{0,2\}}P_1P_3, \Delta Q_{\{0,3\}}P_1P_2, \Delta Q_{\{1,2,3\}}P_0\}
    \]
    and points $Q_{\{1,2\}},Q_{\{1,3\}},Q_{\{2,3\}}$, where $Q_{\{i,j\}}\in \interior{\Delta P_iP_j}$ for $i,j\in \{1,2,3\}\mbox{, }i<j$ and $Q_{\{i,j\}}\preceq Q_{\{1,2,3\}}$. The following conditions are equivalent:
    \begin{enumerate}
        \item \label{1!!} The family $\mathfrak{R}$ intersects nonempty.
        \item \label{2!!} For every $i,j,k\in\{0,1,2,3\}$ such that $0\leq i<j<k\leq 3$ the following holds.
        \[
        \frac{P_iQ_{\{i,j\}}}{Q_{\{i,j\}}P_j}
        \frac{P_jQ_{\{j,k\}}}{Q_{\{j,k\}}P_k}
        \frac{P_kQ_{\{k,i\}}}{Q_{\{k,i\}}P_i}=1
        \]
    \end{enumerate}
\end{proposition}
    \begin{proof}
        Define $C_1=\Delta Q_{\{2,3\}}P_0P_1$, $C_2=\Delta Q_{\{1,3\}}P_0P_2$, $C_3=\Delta Q_{\{1,2\}}P_0P_3$ and
        \[
        \mathfrak{R'}=\{C_1,C_2,C_3\}\cup \mathfrak{R} \setminus \Delta Q_{\{1,2,3\}}P_0.
        \]
        Suppose that condition \eqref{1!!} holds. There exists point $X\in\bigcap\mathfrak{R}$, such that
        \[
        X\in \Delta Q_{\{1,2,3\}}P_0\subset C_i
        \]
        for~$i\in\{1,2,3\}$. By \cite[Theorem~1]{Buba-Brzozowa} used for $\mathfrak{R'}$ we obtain~\ref{2!!}.

        Assume that conditions \eqref{2!!} holds. Then, once again, by \cite[Theorem~1]{Buba-Brzozowa} there exists $X\in\bigcap\mathfrak{R'}$. Let us consider line $P_0X$ and point $Y\in P_0X\cap \Delta P_1P_2P_3$. Then $\Delta P_0Y\subset C_i$ for $i\in \{1,2,3\}$, since $C_i$ is convex and contains points $P_0$ and $Y$. This implies that $Y=Q_{\{1,2,3\}}$, thus $X\in \Delta Q_{\{1,2,3\}}P_0.$
    \end{proof}
    

    This result suggests the possibility of a more general theorem establishing conditions for the non-emptiness of a family of cevians of varying dimensions. In such a theorem, our main result, Theorem~\ref{Main!}, would appear as a special case.

{\bf Acknowledgement.} The author expresses sincere gratitude to Professor Przemysław Koprowski for the valuable discussions, comments, and for initially hypothesizing the concept that subsequently developed into Theorem \eqref{Main!}.

\bibliographystyle{amsplain}

\end{document}